\PassOptionsToPackage{square,comma,numbers,sort&compress}{natbib}  
\documentclass[preprint,12pt]{elsarticle}

\usepackage{natbib}
\setcitestyle{sort&compress}

\usepackage{lineno,hyperref}
\modulolinenumbers[5]

\usepackage{amsmath}
\usepackage{amssymb}
\usepackage{amsthm}

\allowdisplaybreaks

\usepackage[a4paper]{geometry}
\usepackage{graphicx}
\usepackage{caption}
\usepackage{subcaption}
\newcommand{\dd}[1]{\mathrm{d}#1}

\newcommand{\dbxn}[1]{\left(\frac{\dd{}}{\dd{x}}-bx\right)^{(#1)}}
\newcommand{\dbxni}[1]{\left({\dd{}}/{\dd{x}}-bx\right)^{(#1)}}
\newcommand{\daxn}[1]{\left({\dd{}}/{\dd{x}}-2ax\right)^{(#1)}}
\newcommand{\ddx}[0]{\frac{\dd{}}{\dd{x}}}
\newcommand{\ddxn}[1]{\frac{\dd{}^{#1}}{\dd{x^{#1}}}}
\newcommand{\si}[1]{{\sigma}^{(#1)}}

\newcommand{\pHy}[1]{\frac{\partial H_y(x,\theta)}{\partial #1}}
\newcommand{\pEz}[1]{\frac{\partial E_z(x,\theta)}{\partial #1}}

\newcommand{\pt}[3]{\frac{ {\partial}^{#3} {#1}}{\partial {#2}^{#3}}}
\newcommand{\pp}[2]{\frac{{\partial} {#1}}{{\partial} {#2}}}
\newcommand{\pth}[1]{\frac{\theta ^{#1}}{\left( {#1} \right)!}}
\newcommand{\ppj}[3]{\frac{\partial ^{#1} J_z(x,0)}{\partial \theta^{#2} \partial x^{#3}}}

\newtheorem{thm}{Theorem}
\newtheorem{lemma}[thm]{Lemma}

\begin{document}

\begin{frontmatter}

\title{Analytical Solutions of 1D Maxwell's Equations via Infinite-Order Expansions}

\author[Independent Researcher]{David Wei Ge\texorpdfstring{\corref{cor1}}{}}
\ead{gexiaobao.usa@gmail.com}
\cortext[cor1]{Corresponding author}
\address[Independent Researcher]{Retired, formerly with Microsoft, Redmond, USA}

\begin{abstract}
Analytical solutions to Maxwell’s equations are essential for understanding the causal and instantaneous behavior of electromagnetic fields, yet they are challenging to obtain in open-space settings with general initial values and source terms. This work uses an infinite-order estimation scheme that yields closed-form analytical solutions to Maxwell’s equations. The scheme is expressed as general function-to-function transformations that directly map initial values and source terms to the fields. Several case studies demonstrate the method, producing exact solutions for different initial and source configurations. The results provide both theoretical insight and practical benchmarks for computational electromagnetics.
\end{abstract}

\begin{keyword}
Maxwell's equations \sep Analytical solutions \sep Partial differential equations (PDEs) \sep Finite-difference time-domain (FDTD) \sep Wave propagation 
\end{keyword}

\end{frontmatter}

\section{Introduction}
Maxwell’s equations underpin all classical electromagnetics and describe how electromagnetic fields evolve in space and time. While numerical methods, especially the finite-difference time-domain (FDTD) scheme, are widely used to approximate these equations \cite{fdtd:Yee, fdtd:meep, fdtd:Schneider, maxn:Neumann, maxn:Zhou, maxn:Fedeli, fdtd1d:Mastryukov, fdtd:SchneiderDispersion, maxn:Kohlmann}, they inevitably introduce truncation and dispersion errors, grid anisotropy, and occasionally non-physical effects such as faster-than-light propagation \cite{fdtd:SchneiderDispersion, fdtd:Manry}. Achieving closed-form time-domain solutions remains a longstanding challenge \cite{maxs:Oleg, Lipan:24, LeBoudec2024}.

A recent advance by Hongli Yang et al. (\cite{YANG2023126678} 2023) achieved analytical time-domain solutions based on the four-potential wave equations, combining scalar potentials, curls, Laplacians, and integrations of sources. The present work pursues the same goal—exact analytical solutions of Maxwell’s equations for arbitrary analytic initial values and source terms—through a different approach. Starting from the first-order vector form of Maxwell’s equations, we derive infinite-order FDTD-type formulas that remove truncation error entirely, yielding compact analytical expressions. Potential-based approaches introduce gauge freedom; this issue does not arise here.

Central to this formulation is the distinction between mapping—the correspondence between input and output functions—and transformation, the explicit formula that realizes it. Classical Taylor series define self-mappings of analytic functions. The theorems presented here extend that framework by defining mappings from analytic inputs to distinct outputs, enabling exact function-to-function transformations of Maxwell’s equations in the time domain.

The resulting formulas offer two main advantages over conventional numerical schemes: (1) each point in space-time can be computed independently, requiring no computational domain or artificial boundaries, and (2) the true physical properties of the equations are preserved without numerical artifacts. These analytical solutions complement, rather than replace, numerical methods such as FDTD or finite-element time-domain (FETD) schemes, providing rigorous reference fields for verification.

For simplicity, this paper only presents one-dimensional solutions, obtained by collapsing the curls to derivatives in the three-dimensional solution formulas. After presenting and proving the solution theorems, four case studies demonstrate the proposed transformations for various initial and source conditions. The proofs of several auxiliary formulas are given in the Appendix. The approach presented here represents a novel contribution to the analytical study of Maxwell’s equations and a foundation for future research.

\section{The Problem in 1D}
\subsection{Definition of the problem} Consider the following one-dimensional Maxwell’s equations in open space (\cite{maxs:Feynman}: Eqs. (20.16) and (20.18); \cite{fdtd:Schneider}: Eqs. (3.9) and (3.10)):
\begin{equation}
	\label{1d.1}
	\pHy{\theta} = \frac{1}{\eta} \pEz{x}
\end{equation}
\begin{equation}
	\label{1d.2}
	\pEz{\theta} = \eta \pHy{x} - \eta J_z(x,\theta)
\end{equation}
\begin{equation}
	\label{1d.3}
	H_y(x,0) = f_{iH}(x)
\end{equation}
\begin{equation}
	\label{1d.4}
	E_z(x,0) = f_{iE}(x)
\end{equation}
where $\eta = \sqrt{{\mu}/{\epsilon}}$, $\theta = ct$, $c = {1}/{\sqrt{\epsilon \mu}}$, and
\begin{equation*}
	\begin{gathered}
		x,t \in \mathbb{R}; \epsilon, \mu \in \mathbb{R}_{>0};
		f_{iH}, f_{iE} \in C^{\infty}(\mathbb{R},\mathbb{R}); J_z, H_y, E_z \in C^{\infty}(\mathbb{R}^2, \mathbb{R})
	\end{gathered}
\end{equation*}
where $\epsilon$ and $\mu$ are constants, $x$ and $t$ are space and time, respectively, function $J_z (x,\theta)$ is the source, functions $f_{iH} (x)$ and $f_{iE} (x)$  are the initial values. The equations are with dimensionless units \cite{maxn:Kohlmann, fdtd:meep}. For convenience, $\theta$ is referred to as time.

To solve open-space problems using FDTD, artificial boundary conditions are typically introduced to truncate the computational domain, and absorbing layers are employed to suppress reflections and maintain stability. As a result, the solution does not represent true open-space conditions. However, for a short time after a point source is activated—before the fields reach the boundaries—FDTD does produce an accurate approximation of the open-space solution. Building on this observation, the Moving Window FDTD technique estimates the open-space solution in one direction by dynamically shifting the computational domain \cite{maxn:Zhou, fdtd:James}.

\subsection{Definition of the solution}
A function pair \{ $ H_y(x, \theta), E_z(x, \theta) $ \} is a solution to the problem if and only if it satisfies (\ref{1d.1}), (\ref{1d.2}), (\ref{1d.3}) and (\ref{1d.4}). 

\subsection{Solution as a function mapping}
Taking functions $J_z (x,\theta)$, $f_{iH} (x)$ and $f_{iE} (x)$ as the input, function pair \{ $ H_y(x, \theta), E_z(x, \theta) $ \} as the output, the definitions of the problem and its solution define a function-to-function mapping:
\begin{equation}
	\label{1d.fmap}
	\{ J_z(x,\theta), f_{iH}(x), f_{iE}(x) \} \to \{ H_y(x,\theta), E_z(x,\theta) \}
\end{equation}

For practical analysis, we divide the general problem into two fundamental sub-cases.

\textbf{Initial-value problem} :
\begin{equation}
	\begin{gathered}
		\label{1d.fmapIni}
		\{J_z(x,\theta) \equiv 0, f_{iH}(x), f_{iE}(x) \} \to \{ H_y(x,\theta), E_z(x,\theta) \}
	\end{gathered}
\end{equation}

\textbf{Source-driven problem}:
\begin{equation}
	\begin{gathered}
		\label{1d.fmapSrc}
		\{ J_z(x,\theta), f_{iH}(x) \equiv 0, f_{iE}(x) \equiv 0\} \to \{ H_y(x,\theta), E_z(x,\theta) \}
	\end{gathered}
\end{equation}

\section{The 1D solutions}
The theorems in this section provide solution formulas but do not address uniqueness. While the FDTD algorithm yields unique solutions for specified initial values and sources, a full analysis of uniqueness lies beyond the scope of this work.

\begin{thm} \textbf{Initial value solution}. 
	
	The function-to-function transformation defined by formulas (\ref{1d.solIniH}) and (\ref{1d.solIniE}) provides a mapping that satisfies the initial-value problem described by (\ref{1d.fmapIni}) in the range
	\begin{equation*}
		0 \le \theta \le \theta_{max}, 0 \le \lvert x \rvert \le x_{max}
	\end{equation*}
	 provided that the corresponding series converge within this range.			
	\begin{equation}
		\label{1d.solIniH}
		H_y(x,\theta) = f_{iH} + \sum_{n=0}^{\infty} \frac{\theta ^{2(n+1)}}{\left(2(n+1)\right)!} \pt{f_{iH}}{x}{2(n+1)} + \frac{1}{\eta}\sum_{n=0}^{\infty} \frac{\theta ^{2n+1}}{(2n+1)!} \pt{f_{iE}}{x}{2n+1}
	\end{equation}
	\begin{equation}
		\label{1d.solIniE}
		E_z(x,\theta) = f_{iE} + \sum_{n=0}^{\infty} \frac{\theta ^{2(n+1)}}{\left(2(n+1)\right)!} \pt{f_{iE}}{x}{2(n+1)} + \eta\sum_{n=0}^{\infty} \frac{\theta ^{2n+1}}{(2n+1)!} \pt{f_{iH}}{x}{2n+1}
	\end{equation}
	
\end{thm}
\begin{proof}
	It is easy to verify that (\ref{1d.solIniH}) and (\ref{1d.solIniE}) satisfy (\ref{1d.3}) and (\ref{1d.4}). What remains to be shown is that they also satisfy (\ref{1d.1}) and (\ref{1d.2}) under the condition $J_z = 0$.
	
	From (\ref{1d.solIniE}) we have
	\begin{equation*}
		\pp{E_z(x,\theta)}{x} = \sum_{n=0}^{\infty} \pth{{2n}} \pt{f_{iE}}{x}{2n+1}
		 + {\eta}\sum_{n=0}^{\infty} \pth{{2n+1}} \pt{f_{iH}}{x}{2n+2}
	\end{equation*}

	From (\ref{1d.solIniH}) we have
\begin{equation*}
	\pp{H_y(x,\theta)}{\theta} = \sum_{n=0}^{\infty} \pth{{2n+1}} \pt{f_{iH}}{x}{2(n+1)} + \frac{1}{\eta}\sum_{n=0}^{\infty} \pth{{2n}} \pt{f_{iE}}{x}{2n+1}
\end{equation*}

The preceding two equations show that (\ref{1d.solIniH}) and (\ref{1d.solIniE}) satisfy (\ref{1d.1}).

	From (\ref{1d.solIniH}) we have
\begin{equation*}
	\pp{H_y(x,\theta)}{x} = \sum_{n=0}^{\infty} \pth{2n} \pt{f_{iH}}{x}{2n+1} + \frac{1}{\eta} \sum_{n=0}^{\infty} \pth{2n+1} \pt{f_{iE}}{x}{2n+2}
\end{equation*}	
	
	From (\ref{1d.solIniE}) we can obtain
	\begin{equation*}
		\pp{E_z(x,\theta)}{\theta} = \sum_{n=0}^{\infty} \pth{2n+1} \pt{f_{iE}}{x}{2(n+1)} + \eta \sum_{n=0}^{\infty} \pth{2n} \pt{f_{iH}}{x}{2n+1}
	\end{equation*}
		
	The preceding two equations show that (\ref{1d.solIniH}) and (\ref{1d.solIniE}) satisfy (\ref{1d.2}) if $J_z = 0$.
	
	$\therefore$ (\ref{1d.solIniH}) and (\ref{1d.solIniE}) satisfy (\ref{1d.1}), (\ref{1d.2}), (\ref{1d.3}), and (\ref{1d.4}), under the condition $J_z = 0$.
\end{proof}

The following result follows directly from the above theorem.

	\begin{equation}
		\label{1d.solIniH.close}
		H_y(x,\theta) =\frac{1}{2}\left( f_{iH}(x-\theta) -\frac{1}{\eta}f_{iE}(x-\theta) + f_{iH}(x+\theta)+\frac{1}{\eta}f_{iE}(x+\theta)  \right)
	\end{equation}
	\begin{equation}
		\label{1d.solIniE.close}
		E_z(x,\theta) =\frac{1}{2}\left(\eta f_{iH}(x+\theta)+f_{iE}(x+\theta) -\eta f_{iH}(x-\theta)+f_{iE}(x-\theta)  \right)
	\end{equation}

The above formulas are equivalent with d’Alembert formula, see section 2.4.1 in \cite{pdfEvans}. 
Note that the simple formulas (\ref{1d.solIniH.close}) and (\ref{1d.solIniE.close}) apply only to the 1D formulation; higher-dimensional cases are considerably more involved. 

\begin{thm} \textbf{Source driven solution}. 
	
	The function-to-function transformation defined by formulas (\ref{1d.solSrcH}) and (\ref{1d.solSrcE}) provides a mapping that satisfies the source-driven problem described by (\ref{1d.fmapSrc}) in the range
	\begin{equation*}
		0 \le \theta \le \theta_{max}, 0 \le \lvert x \rvert \le x_{max}
	\end{equation*}
	 provided that the corresponding series converge within this range.
	 			
	\begin{equation}
		\begin{gathered}
			\label{1d.solSrcH}
			H_y(x,\theta) =
			\\ - \sum_{n=0}^{\infty} \pth{2n+2} \sum_{m=0}^{n} \ppj{2n+1}{2(n-m)}{2m+1} 
			\\ - \sum_{n=0}^{\infty} \pth{2n+3} \sum_{m=0}^{n} \ppj{2n+2}{2(n-m)+1}{2m+1}
		\end{gathered}
	\end{equation}
	\begin{equation}
		\begin{gathered}
			\label{1d.solSrcE}
			E_z(x,\theta) =
			\\ - \eta \sum_{n=0}^{\infty} \pth{2n+2} \sum_{m=0}^{n} \ppj{2n+1}{2(n-m)+1}{2m}
			\\ - \eta \sum_{n=0}^{\infty} \pth{2n+1} \sum_{m=0}^{n} \ppj{2n}{2(n-m)}{2m}
		\end{gathered}
	\end{equation}	
\end{thm}
\begin{proof}
	It is easy to verify that (\ref{1d.solSrcH}) and (\ref{1d.solSrcE}) satisfy (\ref{1d.3}) and (\ref{1d.4}) under the conditions $f_{iH} = 0$ and $f_{iE} = 0$. What remains to be shown is that they also satisfy (\ref{1d.1}) and (\ref{1d.2}).
	
	From (\ref{1d.solSrcH}) we have
\begin{equation*}
	\begin{gathered}	
		\pp{H_y(x,\theta)}{\theta} =
		\\-\sum_{n=0}^{\infty} \pth{{2n+1}} \sum_{m=0}^{n} \ppj{2n+1}{2(n-m)}{2m+1} 
		\\- \sum_{n=0}^{\infty} \pth{{2n+2}}\sum_{m=0}^{n}\ppj{2n+2}{2(n-m)+1}{2m+1}
	\end{gathered}
\end{equation*}

From (\ref{1d.solSrcE}) we have
\begin{equation*}
	\begin{gathered}
		\frac{1}{\eta} \pp{E_z(x,\theta)}{x} =
		\\-\sum_{n=0}^{\infty} \pth{{2n+1}} \sum_{m=0}^{n} \ppj{2n+1}{2(n-m)}{2m+1}
		\\- \sum_{n=0}^{\infty} \pth{{2n+2}}\sum_{m=0}^{n}\ppj{2n+2}{2(n-m)+1}{2m+1} 
	\end{gathered}
\end{equation*}

The preceding two equations show that (\ref{1d.solSrcH}) and (\ref{1d.solSrcE}) satisfy (\ref{1d.1}).

	From (\ref{1d.solSrcH}) we have
\begin{equation*}
	\begin{gathered}	
		\pp{H_y(x,\theta)}{x} =
		\\ - \sum_{n=0}^{\infty} \pth{2n+2} \sum_{m=0}^{n} \ppj{2n+2}{2(n-m)}{2m+2} 
		\\ - \sum_{n=0}^{\infty} \pth{2n+3} \sum_{m=0}^{n} \ppj{2n+3}{2(n-m)+1}{2m+2}
	\end{gathered}
\end{equation*}
Rearranging the summations to make index $n$ start from 1, we obtain
\begin{equation*}
	\begin{gathered}	
		\pp{H_y(x,\theta)}{x} =
		\\ - \sum_{n=1}^{\infty} \pth{2n} \sum_{m=1}^{n} \ppj{2n}{2(n-m)}{2m} 
		\\ - \sum_{n=1}^{\infty} \pth{2n+1} \sum_{m=1}^{n} \ppj{2n+1}{2(n-m)+1}{2m}
	\end{gathered}
\end{equation*}
Make the two series start from $n=0$ by adding two new series, and we obtain
\begin{equation*}
	\begin{gathered}	
		\pp{H_y(x,\theta)}{x} =
		\\ -\sum_{n=0}^{\infty} \pth{{2n}} \sum_{m=0}^{n} \ppj{2n}{2(n-m)}{2m} 
		+ \sum_{n=0}^{\infty} \pth{{2n}}\pt{J_z(x,0)}{\theta}{2n}
		\\ -\sum_{n=0}^{\infty} \pth{{2n+1}} \sum_{m=0}^{n} \ppj{2n+1}{2(n-m)+1}{2m}
		+ \sum_{n=0}^{\infty} \pth{{2n+1}}\pt{J_z(x,0)}{\theta}{2n+1}
	\end{gathered}
\end{equation*}
The two newly added series form a Taylor series of $J_z(x,\theta)$, and we arrive at
\begin{equation*}
	\begin{aligned}	
		\pp{H_y(x,\theta)}{x} &= J_z(x,\theta)
		\\ &-\sum_{n=0}^{\infty} \pth{2n} \sum_{m=0}^{n} \ppj{2n}{2(n-m)}{2m} 
		\\&-\sum_{n=0}^{\infty} \pth{2n+1} \sum_{m=0}^{n}\ppj{2n+1}{2(n-m)+1}{2m}
	\end{aligned}
\end{equation*}

The preceding equation shows that (\ref{1d.solSrcH}) and (\ref{1d.solSrcE}) satisfy (\ref{1d.2}).

	$\therefore$ (\ref{1d.solSrcH}) and (\ref{1d.solSrcE}) satisfy (\ref{1d.1}), (\ref{1d.2}), (\ref{1d.3}), and (\ref{1d.4}), under the conditions $f_{iH} = 0$ and $f_{iE} = 0$.
\end{proof}

\section{Case studies}

The new theorems, in principle, apply to arbitrary analytic initial values and source terms. In many physical applications, however, sources are defined at discrete locations \cite{maxn:Kohlmann, maxn:Fedeli, fdtd:Schneider}, which renders them non-differentiable. Conventional numerical schemes such as FDTD circumvent this issue by substituting discrete quantities with piecewise differentiable representations. Analytical approaches, on the other hand, require smooth inputs, and the Gaussian function serves as a practical and illustrative choice. Figure \ref{fig:gauss} illustrates this transition from discrete to analytic representation using Gaussian approximations to point sources.

Fundamentally, this limitation reflects a property intrinsic to Maxwell’s equations when treated analytically: physical systems often involve discontinuities that defy strict differentiability. Numerical schemes mask these issues, along with the influence of parameters such as grid spacing, whereas analytical formulations expose them directly. The solution formulas presented here thus open a pathway to studying such effects explicitly, though this lies beyond the present scope.

\begin{figure}[ht]
	\begin{subfigure}{.33\textwidth}
		\includegraphics[width=.82\linewidth]{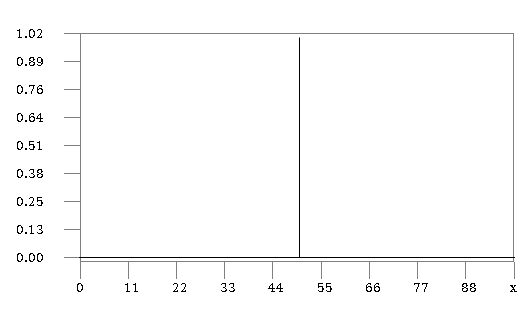}
		\caption{A discrete value}
		\label{fig:pointvalue}
	\end{subfigure}%
	\begin{subfigure}{.33\textwidth}
		\includegraphics[width=.82\linewidth]{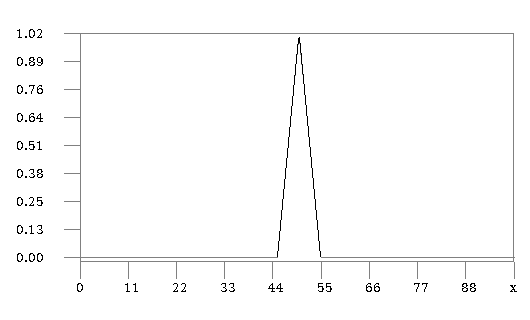}
		\caption{A discrete value in FDTD}
		\label{fig:pointfdtd}
	\end{subfigure}%
	\begin{subfigure}{.33\textwidth}
		\includegraphics[width=.82\linewidth]{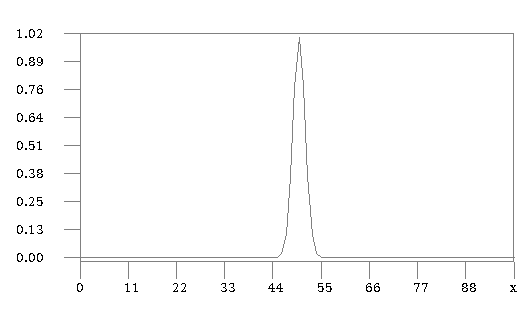}
		\caption{Gaussian function}
		\label{fig:pointgauss}
	\end{subfigure}
	
	\caption{A discrete value at a space point.}
	\label{fig:gauss}
\end{figure}

Each of Fig.{\ref{fig:exa.1}}–{\ref{fig:exa.3} displays ten curves, where each curve represents a snapshot of the analytical solution at a specific time $\theta=ct$. By plotting multiple snapshots together, the figures illustrate the propagation of the solution over time. 
	
The analytical solutions derived in this work are expressed in terms of infinite series. In practical computations, these series must be truncated to finite sums. In all numerical evaluations presented here, the truncation is chosen so that further increasing the number of retained terms does not lead to a detectable change in the computed results, indicating convergence.

In the present implementation, high-precision arithmetic with 4000 decimal digits is employed, and up to 40 000 terms are retained in the series evaluations. This choice ensures robust convergence of the numerical results and serves as a reference implementation rather than an optimized computational strategy.

The development of efficient algorithms for evaluating these series, including optimal truncation criteria and the minimum required numerical precision for a given accuracy, is an important topic for future research. Such advances would enable the construction of standardized and computationally efficient software libraries for practical applications.

\subsection{Case-study 1: initial-value solution}
The following initial values are used.
\begin{equation}
	\label{1d.exa.ini.H}
	f_{iH}(x) = 0
\end{equation}
\begin{equation}
	\label{1d.exa.ini.E}
	f_{iE}(x) = \exp(-ax^2); a \in \mathbb{R}_{>0}
\end{equation}
where $a$ is a constant.

By applying (\ref{f4}) and (\ref{f5}) to (\ref{1d.exa.ini.E}) and substituting the results into (\ref{1d.solIniH}) and (\ref{1d.solIniE}), we obtain the following closed-form analytical solution of Maxwell’s equations for the specified initial values:

\begin{equation}
	\label{1d.exa.sol.ini.H}
	H_y(x,\theta) = -\frac{1}{\eta}\exp({-a(x^2+\theta ^2)}) \sinh(2ax\theta)
\end{equation}
\begin{equation}
	\label{1d.exa.sol.ini.E}
	E_z(x,\theta) = \exp({-a(x^2+\theta ^2)}) \cosh(2ax\theta)
\end{equation}

Formulas (\ref{1d.solIniH.close}) and (\ref{1d.solIniE.close}) also produce the same solution, as expected.

The numerical results obtained from (\ref{1d.exa.sol.ini.H}) and (\ref{1d.exa.sol.ini.E}) with $a=1$ are shown in Fig.\ref{fig:exa.1}.

\begin{figure}[!htb]
	\begin{subfigure}{.5\textwidth}
		\includegraphics[width=.9\linewidth]{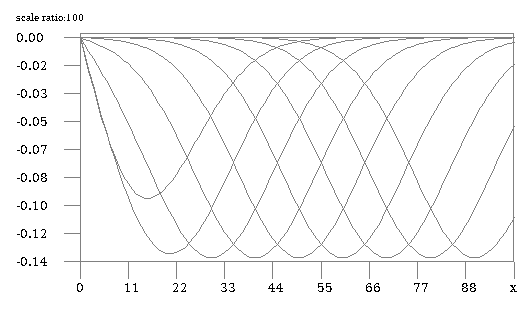}
		\caption{10 snapshots of $H_y$}
		\label{fig:exa.1.H}
	\end{subfigure}%
	\begin{subfigure}{.5\textwidth}
		\includegraphics[width=.9\linewidth]{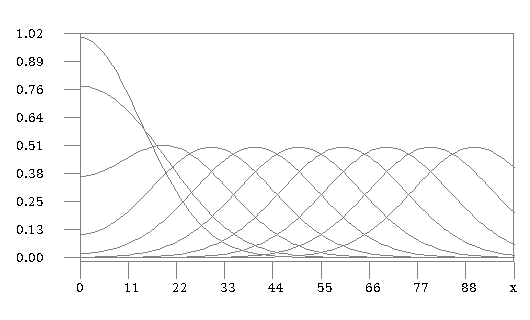}
		\caption{10 snapshots of $E_z$}
		\label{fig:exa.1.E}
	\end{subfigure}
	\caption{Initial value solution (Case-study 1).}
	\label{fig:exa.1}
\end{figure}

\subsection{Case-study 2: source-driven solution}
The following source is used.
\begin{equation}
	\label{1d.exa.src}
	J_z(x,\theta) = \sqrt{a} \exp({-ax^2}) ; a \in \mathbb{R}_{>0}
\end{equation}
where $a$ is a constant.

By applying (\ref{f4}) and (\ref{f5}) to (\ref{1d.exa.src}) and substituting the results into (\ref{1d.solSrcH}) and (\ref{1d.solSrcE}), we obtain the following closed-form analytical solution of Maxwell's equations for the specified source:
\begin{equation}
	\label{1d.exa.sol.src.H}
	H_y(x,\theta) = \exp({-ax^2}) \mathrm{esinh}_1(\sqrt{a}\theta, 2\sqrt{a}x)
\end{equation}
\begin{equation}
	\label{1d.exa.sol.src.E}
	E_z(x,\theta) = -\eta \exp({-ax^2}) \mathrm{eicosh}(\sqrt{a}\theta,2\sqrt{a}x)
\end{equation}
where
\begin{equation}
	\label{esinh1}
	\mathrm{esinh}_1(\xi,\sigma) :=  \sum_{n=0}^{\infty}\frac{(-1)^n\xi ^{2n+2}}{(2n+2)!}\sum_{k=0}^{n}(-1)^kp_{n,k}\sigma ^{2k+1}
\end{equation}
\begin{equation}
	\label{eicosh}
	\mathrm{eicosh}(\xi,\sigma) := \sum_{n=0}^{\infty}\frac{(-1)^n\xi ^{2n+1}}{(2n+1)!} \sum_{k=0}^{n}(-1)^kq_{n,k}\sigma ^{2k}
\end{equation}
where $p_{n,k}$ is given by (\ref{pnk}) and $q_{n,k}$ is given by (\ref{qnk}). Two new special functions, denoted by $\mathrm{esinh_1}$ and $\mathrm{eicosh}$, appear in the solution formulas. The analysis in this work depends only on their explicit series definitions; the specific choice of notation is adopted for convenience.

From the definitions of the special functions $\mathrm{esinh}_1$ and $\mathrm{eicosh}$, it follows directly that the defining series converge for $\lvert \varrho \rvert < 1$ and $ \lvert \xi \rvert <1$. A more detailed analysis of the radius of convergence and the maximal domain of analyticity of these functions is beyond the scope of the present work and will be investigated in future research.

The numerical results obtained from (\ref{1d.exa.sol.src.H}) and (\ref{1d.exa.sol.src.E}) with $a=1$ are shown in Fig.\ref{fig:exa.2}.

\begin{figure}[ht]
	\begin{subfigure}{.5\textwidth}
		\includegraphics[width=.9\linewidth]{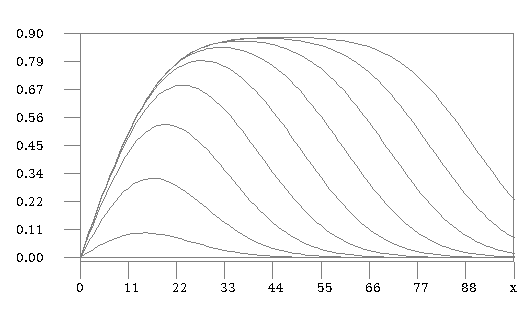}
		\caption{10 snapshots of $H_y$}
		\label{fig:exa.2.H}
	\end{subfigure}%
	\begin{subfigure}{.5\textwidth}
		\includegraphics[width=.9\linewidth]{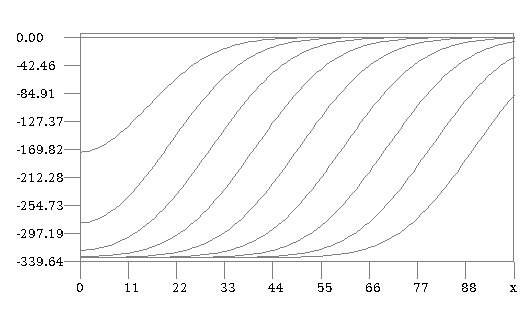}
		\caption{10 snapshots of $E_z$}
		\label{fig:exa.2.E}
	\end{subfigure}
	\caption{Source-driven solution (Case-study 2).}
	\label{fig:exa.2}
\end{figure}

\subsection{Case-study 3: turning off source}
Suppose in case-study 2 at a time $\theta = \theta _0 > 0$, the source is turned off: $J_z(x,\theta) = 0, \forall \theta > \theta _0$. We get an "initial value problem" which uses the following initial-values.
\begin{equation}
	\label{1d.exa.srcini.H}
	f_{iH}(x) = \exp({-ax^2}) S_H(x)
\end{equation} 
\begin{equation}
	\label{1d.exa.srcini.E}
	f_{iE}(x) = -\eta \exp({-ax^2}) S_E(x)
\end{equation}
where
\begin{equation}
	\label{1d.exa.srcini.SH}
	S_{H}(x) = \mathrm{esinh}_1(\sqrt{a}\theta_0, 2\sqrt{a}x)
\end{equation} 
\begin{equation}
	\label{1d.exa.srcini.SE}
	S_{E}(x) = \mathrm{eicosh}(\sqrt{a}\theta_0,2\sqrt{a}x)
\end{equation}

By applying (\ref{f6}) and (\ref{f7}) to (\ref{1d.exa.srcini.H}) and (\ref{1d.exa.srcini.E}), and substituting the results into (\ref{1d.solIniH}) and (\ref{1d.solIniE}), we obtain the following closed-form analytical solution of Maxwell's equations:
\begin{equation}
	\label{1d.exa.sol.srcini.H}
	H_y(x,\theta) = \exp({-a(x^2+\theta^2)}) \sum_{k=0}^{\infty}\left( \theta_{2k}S_{H,2k}-\theta_{2k+1}S_{E,2k+1} \right)
\end{equation}
\begin{equation}
	\label{1d.exa.sol.srcini.E}
	E_z(x,\theta) = \eta \exp({-a(x^2+\theta^2)}) \sum_{k=0}^{\infty}\left( \theta_{2k+1}S_{H,2k+1}-\theta_{2k}S_{E,2k} \right) 
\end{equation}
where
\begin{equation}
	\theta_k = \frac{\theta^k}{k!}
\end{equation}
\begin{equation}
	\label{deriv.SH}
	S_{H,k}(x) = \left( \ddx - 2ax  \right) ^{(k)} S_H(x)
\end{equation}
\begin{equation}
	\label{deriv.SE}
	S_{E,k}(x) = \left( \ddx - 2ax  \right) ^{(k)} S_E(x)
\end{equation}
where $\left( \mathrm{d}/{\mathrm{d}x} - 2ax  \right) ^{(k)}$ is defined by (\ref{ds}). Note that comparing with Case-study 2, the solution here is shifted in time by $\theta_0$.

The numerical results obtained from (\ref{1d.exa.sol.srcini.H}) and (\ref{1d.exa.sol.srcini.E}) with $a=1$ and ${\theta}_0 = 10$ are shown in Fig.\ref{fig:exa.3}.

\begin{figure}[ht]
	\begin{subfigure}{.5\textwidth}
		\includegraphics[width=.9\linewidth]{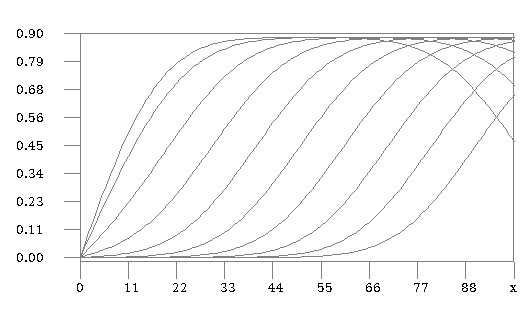}
		\caption{10 snapshots of $H_y$}
		\label{fig:exa.3.H.1}
	\end{subfigure}%
	\begin{subfigure}{.5\textwidth}
		\includegraphics[width=.9\linewidth]{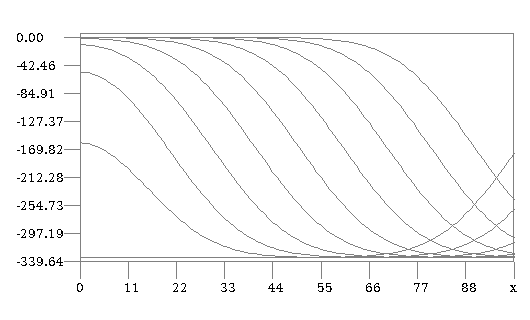}
		\caption{10 snapshots of $E_z$}
		\label{fig:exa.3.E.1}
	\end{subfigure}
	
	\caption{Solution after turning off the source (Case-study 3).}
	\label{fig:exa.3}
\end{figure}

\subsection{Case-study 4: driven by cosine}
The following harmonic source \cite{fdtd:Schneider} is used.
\begin{equation}
	\label{1d.exa.src.cos}
	J_z(x,\theta) = \sqrt{a} \exp({-ax^2})cos(\omega \theta) ; a,\omega \in \mathbb{R}_{>0}
\end{equation}
where $a$ and $\omega$ are two constants.

By applying (\ref{f6}) and (\ref{f7}) to (\ref{1d.exa.src.cos}) and substituting the results into (\ref{1d.solSrcH}) and (\ref{1d.solSrcE}), we obtain the following closed-form analytical solution of Maxwell's equations for the specified source:
\begin{equation}
	\label{1d.exa.sol.src.cos.H}
	H_y(x,\theta) = \exp({-ax^2}) \mathrm{esinh}_1(\sqrt{a}\theta, 2\sqrt{a}x, {\omega}_a)
\end{equation}
\begin{equation}
	\label{1d.exa.sol.src.cos.E}
	E_z(x,\theta) = -\eta \exp({-ax^2}) \mathrm{eicosh}(\sqrt{a}\theta,2\sqrt{a}x, {\omega}_a)
\end{equation}
where
\begin{equation}
	\label{esinh1,f}
	\mathrm{esinh}_1(\xi,\sigma, \omega) :=  \sum_{n=0}^{\infty}{(-1)^n\xi _{2n+2}}\sum_{m=0}^{n}{\omega}^{2(n-m)}\sum_{k=0}^{m}(-1)^kp_{m,k}\sigma ^{2k+1}
\end{equation}
\begin{equation}
	\label{eicosh.f}
	\mathrm{eicosh}(\xi,\sigma, \omega) := \sum_{n=0}^{\infty}{(-1)^n\xi_{2n+1}}\sum_{m=0}^{n}{\omega}^{2(n-m)} \sum_{k=0}^{m}(-1)^kq_{m,k}\sigma ^{2k}
\end{equation}
\begin{equation}
	\label{omegaa}
	{\omega}_a = {\omega}/{\sqrt{a}}; \xi_k=\xi^k/(k!)
\end{equation}

The numerical results obtained from (\ref{1d.exa.sol.src.cos.H}) and (\ref{1d.exa.sol.src.cos.E}) with $a=1$ and $\omega = 1$ are shown in Fig.\ref{fig:exa.4}, showing the propagation of the magnetic wave and the electric wave.

\begin{figure}[ht]
	\begin{subfigure}{.33\textwidth}
		\includegraphics[width=.82\linewidth]{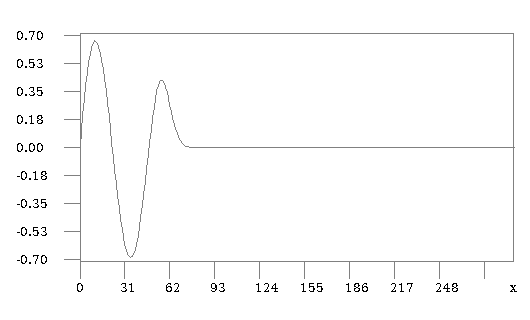}
		\caption{$H_y$ at $\theta=7.5$}
		\label{fig:exa.4.H.1}
	\end{subfigure}%
	\begin{subfigure}{.33\textwidth}
		\includegraphics[width=.82\linewidth]{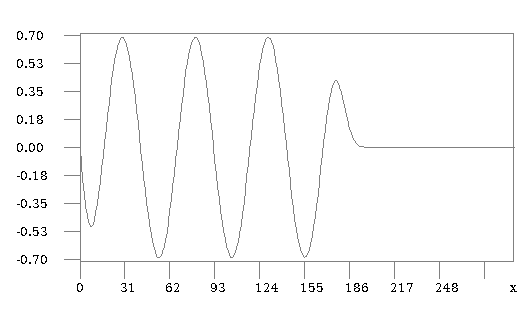}
		\caption{$H_y$ at $\theta=22.5$}
		\label{fig:exa.4.H.2}
	\end{subfigure}%
	\begin{subfigure}{.33\textwidth}
		\includegraphics[width=.82\linewidth]{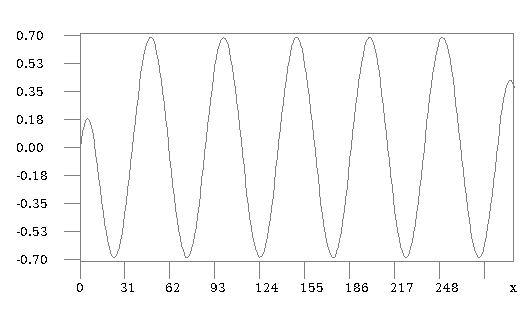}
		\caption{$H_y$ at $\theta=37.5$}
		\label{fig:exa.4.H.3}
	\end{subfigure}
	
	\begin{subfigure}{.33\textwidth}
		\includegraphics[width=.82\linewidth]{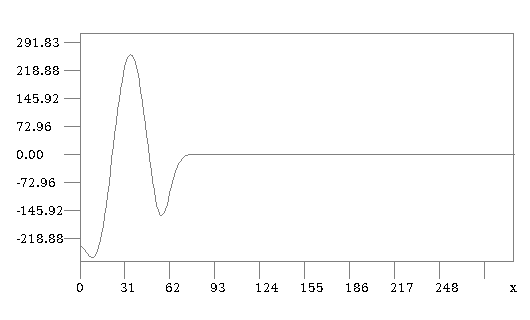}
		\caption{$E_z$ at $\theta=7.5$}
		\label{fig:exa.4.E.1}
	\end{subfigure}%
	\begin{subfigure}{.33\textwidth}
		\includegraphics[width=.82\linewidth]{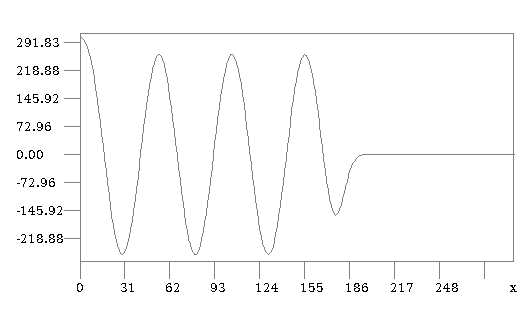}
		\caption{$E_z$ at $\theta=22.5$}
		\label{fig:exa.4.E.2}
	\end{subfigure}%
	\begin{subfigure}{.33\textwidth}
		\includegraphics[width=.82\linewidth]{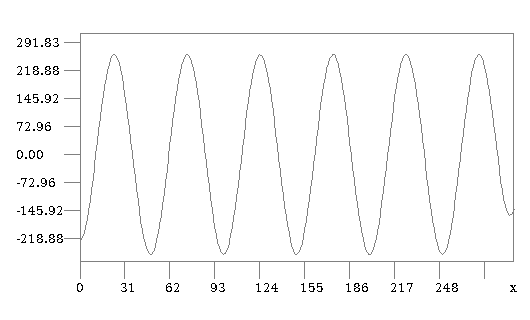}
		\caption{$E_z$ at $\theta=37.5$}
		\label{fig:exa.4.E.3}
	\end{subfigure}
	
	\caption{Solution for cosine source (Case-study 4).}
	\label{fig:exa.4}
\end{figure}

To validate the analytical solution, numerical results obtained using the standard Yee finite-difference time-domain (FDTD) scheme are presented at time $\theta=37.5$.

In all FDTD simulations, the Courant number is set to unity. The spatial and temporal step sizes satisfy the Nyquist–Shannon sampling criterion, and the computational domain is chosen sufficiently large so that boundary effects do not influence the solution at $\theta=37.5$.

A total of ten simulations are performed using progressively refined spatial discretizations. The first simulation employs the same spatial step size as used in the analytical calculations, $\Delta_s=0.125$. The remaining simulations use smaller spatial step sizes given by
\begin{equation*}
	\Delta_s/(2i), i=1,2,...,9
\end{equation*}
The time step $\Delta_{\theta}$ is chosen according to the Courant condition.

The comparison results are shown in Fig.\ref{fig:exa.4.yee}.

Fig.\ref{fig:exa.4.yee.1} presents the numerical solution for $E_z$ obtained using the Yee FDTD scheme together with the corresponding analytical solution, both computed using the same spatial step size. The two solutions are in excellent agreement, with slightly larger discrepancies observed near the beginning and end of the computational domain.

Fig.\ref{fig:exa.4.yee.2} shows that decreasing the spatial step size leads to increased numerical accuracy.

Fig.\ref{fig:exa.4.yee.3} presents the estimation errors corresponding to different spatial step sizes, $\Delta_s=0.125$ and $\Delta_s/(2i), i=1,2,...$ The estimation error is quantified by
\begin{equation*}
	\sum(\Delta_{E_z})^2
\end{equation*}
where $\Delta_{E_z}$ denotes the difference between the numerical and analytical solutions. The errors decrease rapidly as the spatial step size is refined, demonstrating the effect of mesh refinement, and then reach a plateau for sufficiently small step sizes, indicating that further refinement does not lead to a noticeable reduction in error.

\begin{figure}[ht]
	\begin{subfigure}{.33\textwidth}
		\includegraphics[width=.82\linewidth]{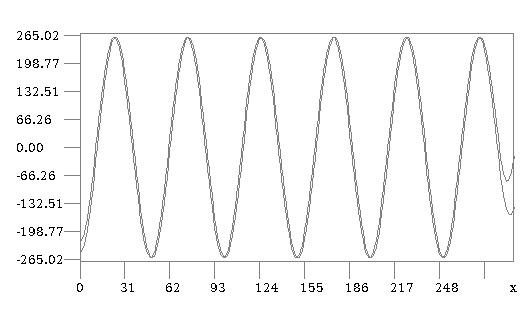}
		\caption{Analytical vs Yee FDTD}
		\label{fig:exa.4.yee.1}
	\end{subfigure}%
	\begin{subfigure}{.33\textwidth}
		\includegraphics[width=.82\linewidth]{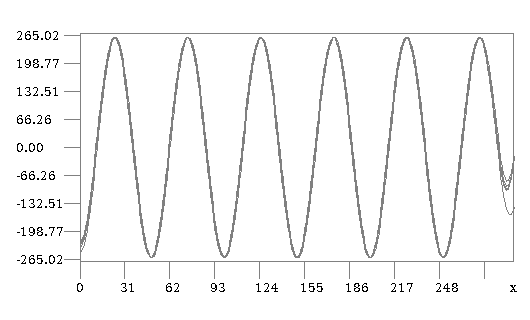}
		\caption{Analytical vs 10 Yee grids}
		\label{fig:exa.4.yee.2}
	\end{subfigure}%
	\begin{subfigure}{.33\textwidth}
		\includegraphics[width=.82\linewidth]{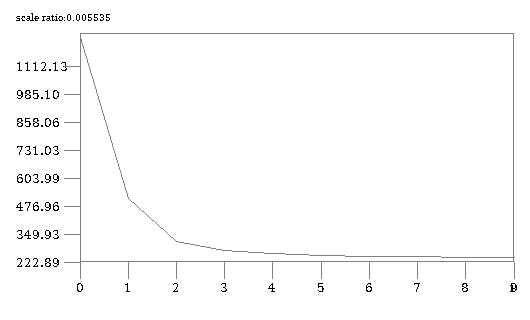}
		\caption{Errors for 10 Yee grids}
		\label{fig:exa.4.yee.3}
	\end{subfigure}
	
	\caption{Comparison of Yee FDTD estimations with the analytical solution.}
	\label{fig:exa.4.yee}
\end{figure}

\section{Discussion}
 Formulas (\ref{1d.solIniH}), (\ref{1d.solIniE}), (\ref{1d.solSrcH}), and (\ref{1d.solSrcE}) extend the classical Taylor framework by enabling mappings from given input functions to potentially different output functions. This generalization broadens the scope from self-mapping of analytic functions to exact time-domain solutions of Maxwell’s equations for arbitrary initial values and sources. The distinction is illustrated by comparing the closed-form formulas (\ref{1d.solIniH.close}) and (\ref{1d.solIniE.close}) with their corresponding series representations given by (\ref{1d.solIniH}) and (\ref{1d.solIniE}).

The outputs of (\ref{1d.solSrcH}) and (\ref{1d.solSrcE}) may or may not belong to the class of known or elementary functions. For example, when a Gaussian function is used as input, Formulas (\ref{1d.solIniH}) and (\ref{1d.solIniE}) yield the elementary functions given in Eqs. (\ref{1d.exa.sol.ini.H}) and (\ref{1d.exa.sol.ini.E}). Conversely, when elementary functions such as those in Eqs. ({\ref{1d.exa.src}}) and ({\ref{1d.exa.src.cos}}) are used as inputs, Formulas (\ref{1d.solSrcH}) and (\ref{1d.solSrcE}) produce new, non-elementary functions—specifically, those in Eqs. (\ref{1d.exa.sol.src.H}), (\ref{1d.exa.sol.src.E}), (\ref{1d.exa.sol.src.cos.H}), and (\ref{1d.exa.sol.src.cos.E})—which, to the author’s knowledge, have not previously appeared in the literature.

In general, one cannot expect the transformation formulas (\ref{1d.solSrcH}) and (\ref{1d.solSrcE}) to yield only known functions. Instead, they define a broader class of solutions whose analytic structure is determined by the properties of the transformation itself.

The proposed analytical formulas offer two key advantages over conventional numerical estimation methods. First, the calculation of the solution at any single point in space and time is fully independent of other points. As a result, there is no need to define a computational domain, apply artificial boundary conditions, or construct absorbing layers—simplifying implementation and avoiding related artifacts.

Second, the analytical solutions preserve the true physical characteristics of the underlying equations. In contrast, finite-difference time-domain (FDTD) methods often introduce numerical artifacts such as dispersion, grid anisotropy, and even non-physical effects like faster-than-light propagation \cite{fdtd:SchneiderDispersion, fdtd:Manry}. These distortions are entirely avoided by using exact function-to-function transformations.

Although the analytical solutions presented in this paper offer a level of precision unattainable by conventional estimation algorithms, they are not intended to replace numerical methods in most engineering applications. Numerical approaches such as FDTD remain essential for handling problems involving finite domains, complex boundary conditions, irregular terrains \cite{maxn:Kohlmann}, or anisotropic media.

One notable advantage of numerical methods is their flexibility: estimation algorithms are typically independent of the specific initial values and source terms. In contrast, analytical solutions require explicit, differentiable functional forms — conditions not always met in practical applications.

However, these two approaches need not be viewed as mutually exclusive. By using numerical estimations to compute the partial derivatives in the analytical expressions (or in analogous three-dimensional versions not covered in this paper), one can construct high-order FDTD algorithms. The author has implemented such hybrid schemes in 3D using both the Yee grid and a non-staggered grid; in both cases, the expected increases in precision were observed.

\section{Future work}
\label{futurework}
Several directions for future work are outlined below. The items are not ordered by importance.
\begin{itemize}
	\item \textbf{Uniqueness of the solution.}
	
	A rigorous treatment of the uniqueness of the analytical solution remains to be established. One possible approach is to assume the existence of two solutions $(E_{z1},H_{y1})$ and $(E_{z2},H_{y2})$ corresponding to the same initial conditions and source term. The differences
	\begin{equation*}
		\Delta_{E_z}=E_{z1}-E_{z2}, \Delta_{H_y}=H_{y1}-H_{y2}
	\end{equation*}
	then satisfy the source-free one-dimensional Maxwell’s equations with zero initial conditions. Under suitable regularity assumptions, for example
	\begin{equation*}
		\Delta_{E_z}(x,0), \Delta_{H_y}(x,0) \in C^\infty
	\end{equation*}
	it is expected that the only solution is the trivial one, implying uniqueness of the solution.
	
	\item \textbf{Extension to three dimensions.}
	
	An important direction for future research is the extension of the present one-dimensional analytical framework to three-dimensional Maxwell’s equations. Conceptually, this involves replacing the spatial derivatives appearing in the one-dimensional formulation with the corresponding curl operators. The extension to three dimensions introduces additional mathematical and computational challenges, including vectorial coupling and geometric complexity. Preliminary results in this direction have been obtained and will be reported elsewhere.
	
	\item \textbf{Notation for the new family of special functions.}
	
	When studying analytical solutions of Maxwell’s equations involving Gaussian functions, in one or higher spatial dimensions, a family of special functions defined by series combining trigonometric, hyperbolic, and exponential structures naturally arises. In the present work, a simple notation is adopted for convenience. A more systematic discussion of notation and classification for this family of functions may be of interest and is left for future work.	
	
	\item \textbf{Non-smooth source terms and initial conditions.}

The present analytical framework assumes sufficiently smooth source terms and initial conditions, which may not always be satisfied in practical physical applications. An important direction for future work is the treatment of non-smooth inputs. Possible approaches include the use of regularization techniques to construct smooth approximations of the data, or the incorporation of numerical differentiation and approximation schemes to estimate the required derivatives in a stable manner.
	
	\item \textbf{Relation to the Method of Moments.}
	
Although the present approach is formulated in an analytical setting and validated using finite-difference time-domain simulations, its structure is related to numerical estimation techniques. It may therefore be of interest to investigate possible conceptual or methodological connections between the present framework and the classical Method of Moments. A systematic study of such relationships, if any, is left for future work.

\end{itemize} 

\appendix

\section{Formulas used by the case-studies}

To apply the extended Taylor series presented in this paper, one must obtain analytical expressions for the derivatives of the initial values and the source term. For the case studies presented in this paper, the required derivatives are produced by formulas (\ref{f6}), (\ref{f7}), (\ref{f4}), and (\ref{f5}). These new formulas are proved in this appendix.

\subsection{The definitions}
For two integers, $n$ and $k$, $n \ge k \ge 0$, let's define the following coefficients.

Odd-binomial coefficients:
\begin{align}
	\label{pnk}
	p_{n,k} := \frac{(2n+1)!}{(n-k)!(2k+1)!}
\end{align}

Even-binomial coefficients:
\begin{align}
	\label{qnk}
	q_{n,k} := \frac{(2n)!}{(n-k)!(2k)!}
\end{align}

Let's use brackets in superscripts to define the following derivative operator:
\begin{equation}
	\label{ds}
	\left(\frac{\dd{}}{\dd{x}}-s\right)^{(n)} := \sum_{k=0}^{n}\binom{n}{k}(-1)^k{s^k}{\frac{\dd{}^{n-k}}{\dd{x^{n-k}}}}
\end{equation}

For a real value, the brackets in superscripts do nothing, the result is a power of the value:
\begin{equation*}
	s^{(n)} = s^n
\end{equation*}

The set of operators $(\mathrm{d}/{\mathrm{d}x}-s)^{(n)}$ is treated only under addition, forming a commutative structure. Throughout this paper, we restrict ourselves to linear combinations and regroupings, without products. Note that the following relation does not necessarily hold: $\left({\dd{}}/{\dd{x}}-s\right)^{(n+1)}= \left({\dd{}}/{\dd{x}}-s\right) \left({\dd{}}/{\dd{x}}-s\right)^{(n)} $.

\subsection{The Formulas}
\begin{lemma}\label{lemma1} For all $x,b \in \mathbb{R}, k \in \mathbb{Z}_{> 0}$ we have
	\begin{equation}
		\label{f1}
		\begin{gathered}
		\ddx\dbxn{k} = 
		\\ \dbxn{k+1} + bx\dbxn{k} - kb\dbxn{k-1}
		\end{gathered}
	\end{equation}
\end{lemma}
\begin{proof}
The identity follows by starting from the left-hand side of (\ref{f1}) and applying a proper regrouping of the summations, which yields the right-hand side.	

By applying the differential operator $\mathrm{d}/{\mathrm{d}x}$ and using the definition (\ref{ds}), we can form 3 series of the same index range. These 3 series can be further combined into the following 2 series.
	\begin{equation*}	
	\begin{gathered}
		\mathrm{LHS} = \dbxn{k+1} + (-1)^k(bx)^{k+1}  
		+ \sum_{m=1}^{k}(-1)^m\binom{k}{m}mb^mx^{m-1}\ddxn{k-m}
		\\+ \sum_{m=0}^{k-1}(-1)^{m}\binom{k}{m} (bx)^{m+1}\ddxn{k-m}
	\end{gathered}
	\end{equation*}
	The range of the last series grows to $k$ by absorbing the term $(-1)^k(bx)^{k+1}$. By the definition (\ref{ds}), the last series can be written as $bx\dbxni{k}$, we obtain
	\begin{equation*}	
	\begin{gathered}
		\mathrm{LHS} = \dbxn{k+1}
		+ \sum_{m=1}^{k}(-1)^m\binom{k}{m}mb^mx^{m-1}\ddxn{k-m}
		 + bx\dbxn{k}
	\end{gathered}
	\end{equation*}
	By extracting factor $bk$ from the series, shifting the index range of $m$ downwards, and applying  $({(k-1)!(m+1)})/({(k-1-m)!(m+1)!})=\binom{k-1}{m} $, we obtain  
	\begin{equation*}	
	\begin{gathered}
		\mathrm{LHS} = \dbxn{k+1} 
		- bk\sum_{m=0}^{k-1}(-1)^m\binom{k-1}{m}(bx)^{m}\ddxn{k-1-m}
		\\+ bx\dbxn{k}
	\end{gathered}
\end{equation*}
	
	By the definition (\ref{ds}), the series can be written as $\dbxni{k-1}$, we arrive at
	\begin{equation*}	
	\begin{gathered}
		\mathrm{LHS} = \dbxn{k+1} 
		- bk\dbxn{k-1}
		+ bx\dbxn{k}
		= \mathrm{RHS}
	\end{gathered}
	\end{equation*}
	This completes the proof.
\end{proof}

\begin{lemma}\label{lemma2} For all $\sigma \in \mathbb{R} \cup \{{\dd{}}/{\dd{x}}-bx\}, a,b,x \in \mathbb{R}, n \in \mathbb{Z}_{>0}$ we have
	\begin{equation}
		\label{f2}
		\begin{gathered}
			\sum_{k=0}^{n}(-1)^k{a^{n-k}}q_{n,k}{\sigma}^{(2k+1)}-\sum_{k=1}^{n}(-1)^k4k{a^{n-k+1}}q_{n,k}{\sigma}^{(2k-1)}
			\\= \sum_{k=0}^{n}(-1)^k{a^{n-k}}p_{n,k}{\sigma}^{(2k+1)}
		\end{gathered}
	\end{equation}
\end{lemma}
\begin{proof}
Beginning with the left-hand side of (\ref{f2}), we reorganize the terms step by step until the right-hand side is reached.
	
Applying definition (\ref{qnk}) to the left-hand side of (\ref{f2}), isolating the term $k=n$, reindexing the remaining series, and using the identity	
	\begin{equation*}
	\frac{(2n)!}{(n-k)!(2k)!}+\frac{4(k+1)(2n)!}{(n-k-1)!(2k+2)!}=\frac{(2n+1)!}{(n-k)(2k+1)!}=p_{n,k}
	\end{equation*}
	 we obtain
	\begin{equation*}	
	\begin{gathered}
		\mathrm{LHS} = (-1)^n\si{2n+1} 
		 + \sum_{k=0}^{n-1}(-1)^ka^{n-k}p_{n,k} \si{2k+1}
	\end{gathered}
	\end{equation*}	
	Because $p_{n,n}=1$, the series can absorb the term $(-1)^n\si{2n+1}$, and its index range grows to $k=n$. We arrive at
	\begin{equation*}	
	\begin{gathered}
		\mathrm{LHS} =  \sum_{k=0}^{n}(-1)^ka^{n-k}p_{n,k} \si{2k+1} = \mathrm{RHS}
	\end{gathered}
	\end{equation*}	
	This completes the verification of equation (\ref{f2}).
\end{proof}

\begin{lemma}\label{lemma3} For all $\sigma \in \mathbb{R} \cup \{{\dd{}}/{\dd{x}}-bx\}, a,b,x \in \mathbb{R}, n \in \mathbb{Z}_{\ge 0}$ we have
	\begin{equation}
		\label{f3}
		\begin{gathered}
			\sum_{k=0}^{n}(-1)^k{a^{n-k}}p_{n,k}({\sigma}^{(2k+2)}-2(2k+1)a{\sigma}^{(2k)})
		\\	= \sum_{k=0}^{n+1}(-1)^{k+1}{a^{n+1-k}}q_{n+1,k}{\sigma}^{(2k)}
		\end{gathered}
	\end{equation}
\end{lemma}
\begin{proof}
Starting from the left-hand side of (\ref{f3}), a suitable rearrangement of terms yields the expression on the right-hand side.
	
	Rewriting the left hand side of (\ref{f3}) in two series:
\begin{equation*}	
	\begin{gathered}
		\mathrm{LHS} = 
	\sum_{k=0}^{n}(-1)^k{a^{n-k}}p_{n,k}{\sigma}^{(2k+2)}
	- \sum_{k=0}^{n}(-1)^k{a^{n+1-k}}p_{n,k}2(2k+1){\sigma}^{(2k)}
	\end{gathered}
\end{equation*}	
	By separating the last term ($k=n$) from the first series, and separating the first term ($k=0$) from the second series, we obtain
\begin{equation*}	
	\begin{gathered}
		\mathrm{LHS} =  (-1)^np_{n,n}{\sigma}^{(2n+2)} - a^{n+1}p_{n,0}2 
		\\ +\sum_{k=0}^{n-1}(-1)^k{a^{n-k}}p_{n,k}{\sigma}^{(2k+2)}
		- \sum_{k=1}^{n}(-1)^k{a^{n+1-k}}p_{n,k}2(2k+1){\sigma}^{(2k)}
	\end{gathered}
\end{equation*}	
	Because $p_{n,0}={(2n+1)!}/{n!}$ and $p_{n,n}=1$, we have
\begin{equation*}	
	\begin{gathered}
		\mathrm{LHS} =  (-1)^n{\sigma}^{(2n+2)} - a^{n+1}2\frac{(2n+1)!}{n!} 
		\\ +\sum_{k=0}^{n-1}(-1)^k{a^{n-k}}p_{n,k}{\sigma}^{(2k+2)}
		- \sum_{k=1}^{n}(-1)^k{a^{n+1-k}}p_{n,k}2(2k+1){\sigma}^{(2k)}
	\end{gathered}
\end{equation*}	
Combining the two series and applying identity
\begin{equation*}
 p_{n,k}+p_{n,k+1}2(2k+3)=\frac{(2n+2)!}{(n-k)!(2k+2)!}
\end{equation*}
we obtain
\begin{equation*}	
	\begin{gathered}
		\mathrm{LHS} =  (-1)^n{\sigma}^{(2n+2)} - a^{n+1}2\frac{(2n+1)!}{n!} 
		 +\sum_{k=0}^{n-1}(-1)^k{a^{n-k}} \frac{(2n+2)!}{(n-k)!(2k+2)!} {\sigma}^{(2k+2)}
	\end{gathered}
\end{equation*}	
By shifting the index range from $k=0,1,...,n-1$ to $k=1,2,...,n$, the two terms, $(-1)^n{\sigma}^{(2n+2)}$ and $a^{n+1}2{(2n+1)!}/{n!}$, can be absorbed by the series with index $k=n+1$ and $k=0$ respectively. The index range of the series becomes $k=0,1,...,n+1$. Because ${(2n+2)!}/({(n+1-k)!(2k)!})=q_{n+1,k}$, we arrive at	
\begin{equation*}	
	\begin{gathered}
		\mathrm{LHS} =  \sum_{k=0}^{n+1}(-1)^{k+1}{a^{n+1-k}} q_{n+1,k} {\sigma}^{(2k)}=\mathrm{RHS}
	\end{gathered}
\end{equation*}	
Thus the left-hand side transforms into the right-hand side, as claimed. 
\end{proof}

\begin{thm} For all $x,a \in \mathbb{R}, n \in \mathbb{Z}_{\ge 0}, f \in C^{2n+1}(\mathbb{R},\mathbb{R})$ we have
	\begin{equation}
		\label{f6}
		\begin{gathered}
			\ddxn{2n} (\exp({-ax^2})f(x))
			\\= \exp({-ax^2})\sum_{k=0}^{n}q_{n,k}(-a)^{n-k}\daxn{2k}f(x)
		\end{gathered}
	\end{equation} 
	\begin{equation}
		\label{f7}
		\begin{gathered}
			\ddxn{2n+1} (\exp({-ax^2})f(x))
			\\= \exp({-ax^2})\sum_{k=0}^{n}p_{n,k}(-a)^{n-k}{\daxn{2k+1}}f(x)
		\end{gathered}
	\end{equation} 
\end{thm}
\begin{proof}
	We will induct on n, using the 3 lemmas: (\ref{f1}), (\ref{f2}), and (\ref{f3}).
	
	\textbf{Base case ($n=0$)}: It is easy to see that for $n=0$ (\ref{f6}) and (\ref{f7}) hold.
	
	\textbf{Inductive Hypothesis ($n=m$)}: Assume for some $m \in \mathbb{Z}_{\ge 0}$ (\ref{f6}) and (\ref{f7}) hold.
	
	\textbf{Inductive Step}: Use the following notation to make the formulas compact:
	\begin{equation*}
		\Psi^{k} = \daxn{k}
	\end{equation*}
	From (\ref{f7}) and (\ref{f1}) we may deduce the following equation.
	\begin{equation*}
		\begin{gathered}
			\ddxn{2(m+1)}(\exp({-ax^2})f(x)) 
			\\ = \exp({-ax^2})\left(\sum_{k=0}^{m}(-1)^{m-k}p_{m,k}a^{m-k}\left(\Psi^{2k+2}-(2k+1)2a\Psi^{2k} \right) \right)f(x)
		\end{gathered}
	\end{equation*}
	We may extract the factor $(-1)^m $ out of the series. Because $(-1)^{-k}=(-1)^{k} $, the equation becomes
	\begin{equation*}
		\begin{gathered}
			\ddxn{2(m+1)}(\exp({-ax^2})f(x))
			\\= \exp({-ax^2})(-1)^m\left(\sum_{k=0}^{m}(-1)^kp_{m,k}a^{m-k}\left(\Psi^{2k+2}-(2k+1)2a\Psi^{2k} \right) \right)f(x)
		\end{gathered}
	\end{equation*}
	By substituting the right hand side of (\ref{f3}) into the above equation, we obtain
	\begin{equation*}
		\begin{gathered}
			\ddxn{2(m+1)}(\exp({-ax^2})f(x))
			\\ =\exp({-ax^2})(-1)^m\left(\sum_{k=0}^{m+1}(-1)^{k+1}q_{m+1,k}a^{m+1-k}\Psi^{2k} \right)f(x)
		\end{gathered}
	\end{equation*}
	By moving the factor $(-1)^m$ back into the series we arrive at
	\begin{equation*}
		\begin{gathered}
			\ddxn{2(m+1)}(\exp({-ax^2})f(x))
			\\= \exp({-ax^2})\left(\sum_{k=0}^{m+1}q_{m+1,k}(-a)^{m+1-k}\Psi^{2k} \right)f(x)
		\end{gathered}
	\end{equation*}
	The preceding equation shows that (\ref{f6}) holds for $n=m+1$.

Since we have established that (\ref{f6}) holds for $n=m+1$, applying (\ref{f6}) with $n=m+1$ yields
	\begin{equation*}
		\begin{gathered}
			\ddx{}\ddxn{2(m+1)}(\exp({-ax^2})f(x)) = \ddx \exp({-ax^2})\left(\sum_{k=0}^{m+1}q_{m+1,k}(-a)^{m+1-k}\Psi^{2k} \right)f(x)
		\end{gathered}
	\end{equation*}
	Applying the differential operation $\mathrm{d}/{\mathrm{d}x}$ on both side, we obtain
	\begin{equation*}
		\begin{gathered}
			\ddxn{2(m+1)+1}(\exp({-ax^2})f(x)) 
			\\= -2ax \exp({-ax^2})\left(\sum_{k=0}^{m+1}q_{m+1,k}(-a)^{m+1-k}\Psi^{2k} \right)f(x)
			\\ + \exp({-ax^2})\ddx \left(\sum_{k=0}^{m+1}q_{m+1,k}(-a)^{m+1-k}\Psi^{2k} \right)f(x)
		\end{gathered}
	\end{equation*}
	By separating the first term ($k=0$) from both series, we obtain
	\begin{equation*}
		\begin{gathered}
			\ddxn{2(m+1)+1}(\exp({-ax^2})f(x)) = -2ax \exp({-ax^2})q_{m+1,0}(-a)^{m+1} f(x)
			\\ -2ax \exp({-ax^2})\left(\sum_{k=1}^{m+1}(-a)^{m+1-k}\Psi^{2k}q_{m+1,k} \right)f(x)
			\\ + \exp({-ax^2})q_{m+1,0}(-a)^{m+1} \ddx f(x)
			\\ + \exp({-ax^2}) \left(\sum_{k=1}^{m+1}q_{m+1,k}(-a)^{m+1-k}\ddx \Psi^{2k} \right)f(x)
		\end{gathered}
	\end{equation*}
	By combining the two separated terms, we obtain
	\begin{equation*}
		\begin{gathered}
			\ddxn{2(m+1)+1}(\exp({-ax^2})f(x)) = \exp({-ax^2})q_{m+1,0}(-a)^{m+1}\left(\ddx -2ax \right) f(x)
			\\ -2ax \exp({-ax^2})\left(\sum_{k=0}^{m+1}q_{m+1,k}(-a)^{m+1-k}\Psi^{2k} \right)f(x)
			\\ + \exp({-ax^2}) \left(\sum_{k=1}^{m+1}q_{m+1,k}(-a)^{m+1-k}\ddx \Psi^{2k} \right)f(x)
		\end{gathered}
	\end{equation*}
	By substituting the right hand side of (\ref{f1}) into the second series, the first one can be canceled. We obtain
	\begin{equation*}
	\begin{gathered}
		\ddxn{2(m+1)+1}(\exp({-ax^2})f(x)) 
		\\= \exp({-ax^2}) \left( (-1)^{m+1}q_{m+1,0}a^{m+1}\left(\ddx -2ax \right) \right)f(x)
		\\ + \exp({-ax^2}) \left(\sum_{k=1}^{m+1}q_{m+1,k}(-a)^{m+1-k}\Psi^{2k+1} \right)f(x)
		\\ - \exp({-ax^2}) \left(\sum_{k=1}^{m+1}q_{m+1,k}(-a)^{m+1-k}4ka\Psi^{2k-1} \right)f(x)
	\end{gathered}
\end{equation*}
	The first series can absorb the first term, and its starting index becomes 0. We obtain  
	\begin{equation*}
	\begin{gathered}
		\ddxn{2(m+1)+1}(\exp({-ax^2})f(x)) 
		\\= \exp({-ax^2}) \left(\sum_{k=0}^{m+1}q_{m+1,k}(-a)^{m+1-k}\Psi^{2k+1} \right)f(x)
		\\ - \exp({-ax^2}) \left(\sum_{k=1}^{m+1}q_{m+1,k}(-a)^{m+1-k}4ka\Psi^{2k-1} \right)f(x)
	\end{gathered}
\end{equation*}
By moving the factor $(-1)^{m+1}$ out of the two series, we obtain	
	\begin{equation*}
		\begin{gathered}
			\ddxn{2(m+1)+1}(\exp({-ax^2})f(x)) 
			= \exp({-ax^2})(-1)^{m+1} \Bigg(\sum_{k=0}^{m+1}(-1)^ka^{m+1-k}q_{m+1,k} \Psi^{2k+1} 
			 \\- \sum_{k=1}^{m+1}(-1)^ka^{m+2-k} 4k q_{m+1,k} \Psi^{2k-1}  \Bigg)f(x)
		\end{gathered}
	\end{equation*}	
	By substituting the right hand side of (\ref{f2}) with $n=m+1$ into the above, we obtain
	\begin{equation*}
	\begin{gathered}
		\ddxn{2(m+1)+1}(\exp({-ax^2})f(x))
		\\= \exp({-ax^2})(-1)^{m+1}\left(\sum_{k=0}^{m+1}(-1)^kp_{m+1,k}a^{m+1-k}\Psi^{2k+1} \right)f(x)
	\end{gathered}
\end{equation*}	
	By moving the factor $(-1)^{m+1}$ back into the series we arrive at
	\begin{equation*}
		\begin{gathered}
			\ddxn{2(m+1)+1}(\exp({-ax^2})f(x))
			\\ = \exp({-ax^2})\left(\sum_{k=0}^{m+1}p_{m+1,k}(-a)^{m+1-k}\Psi^{2k+1} \right)f(x)
		\end{gathered}
	\end{equation*}
	The preceding equation shows that (\ref{f7}) holds for $n=m+1$.
	
	$\therefore$ By the principle of induction, (\ref{f6}) and (\ref{f7}) hold for all $n\in \mathbb{Z}_{\ge 0}$.
\end{proof}

\subsection{Relations to the Hermite polynomials}

Studying the relations between the above formulas and known results may shed light on their internal consistency and equivalence. This section focuses in particular on their connection with the Hermite polynomials.

Let $f(x) \equiv 1$, (\ref{f6}) and (\ref{f7}) reduce to (\ref{f4}) and (\ref{f5}):
\begin{equation}
	\label{f4}
	\begin{gathered}
		\ddxn{2n}\exp({-ax^2})
		= \exp({-ax^2})(-1)^{n}a^n \sum_{k=0}^{n}(-1)^kq_{n,k}{(2\sqrt{a}x)}^{2k}  \\
		\forall x,a \in \mathbb{R}, n \in \mathbb{Z}_{\ge 0}
	\end{gathered}
\end{equation}
\begin{equation}
	\label{f5}
	\begin{gathered}
		\ddxn{2n+1}\exp({-ax^2})
		= \exp({-ax^2})(-1)^{n+1}{\sqrt{a}}^{2n+1} \sum_{k=0}^{n}(-1)^{k}p_{n,k}{(2\sqrt{a}x)}^{2k+1}  \\
		\forall x,a \in \mathbb{R}, n \in \mathbb{Z}_{\ge 0}
	\end{gathered}
\end{equation} 
Let $a=1$. In this case, formulas (\ref{f4}) and (\ref{f5}) are equivalent to the standard definition of the Hermite polynomials \cite{NIST:DLMF}. Specifically, Eq. (18.5.13) in \cite{NIST:DLMF} can be written as two separate expressions corresponding to the even and odd indices:
\begin{equation*}
	H_{2n}(x)=(2n)!\sum_{l=0}^{n}\frac{(-1)^l(2x)^{2n-2l}}{l!(2n-2l)!}
\end{equation*}
\begin{equation*}
	H_{2n+1}(x)=(2n+1)!\sum_{l=0}^{n}\frac{(-1)^l(2x)^{2n+1-2l}}{l!(2n+1-2l)!}
\end{equation*}
Letting
\begin{equation*}
	k=n-l
\end{equation*}
and applying the definitions (\ref{pnk}) and (\ref{qnk}), the Hermite polynomials can be written in terms of the even binomial coefficients $q_{n,k}$ and the odd binomial coefficients $p_{n,k}$:
\begin{equation*}
	H_{2n}(x)=\sum_{k=0}^{n}(-1)^{n-k}(2x)^{2k}q_{n,k}
\end{equation*}
\begin{equation*}
	H_{2n+1}(x)=\sum_{k=0}^{n}(-1)^{n-k}(2x)^{2k+1}p_{n,k}
\end{equation*}
Using these relations, the results (\ref{f4}) and (\ref{f5}) with $a=1$ can be expressed explicitly in terms of the Hermite polynomials:
\begin{equation*}
	\ddxn{2n}\exp({-x^2})
	= (-1)^n\exp({-x^2})H_{2n}(x)
\end{equation*}
\begin{equation*}
	\ddxn{2n+1}\exp({-x^2})
	= (-1)^{n+1}\exp({-x^2})H_{2n+1}(x)
\end{equation*}

\section*{Acknowledgment}
I thank Dragan Redžić and Wim Vegt for their valuable feedback and constructive comments and suggestions on a draft. I also thank the anonymous reviewers for their professional insights and helpful suggestions, which substantially improved the quality of this manuscript.

Declaration of generative AI and AI-assisted technologies in the manuscript preparation process:

During the preparation of this manuscript the author used ChatGPT to assist polishing the English language. After using this tool, the author reviewed and edited the content as necessary and takes full responsibility for the content of the published article.

\bibliographystyle{elsarticle-num}
\bibliography{solution2}

\end{document}